\documentclass[12pt]{article}

\usepackage[english]{babel}
\usepackage[cp1251]{inputenc}
\usepackage[T2A]{fontenc}

\usepackage{amssymb}
\usepackage{amsmath}
\usepackage{amsthm} 
\usepackage{mathtext}
\usepackage{amsfonts}
\usepackage{bbm}

\newcommand{\refpar}{Sect.~}

\newcommand{\lda}{\lambda}
\DeclareMathOperator{\supp}{supp}
\DeclareMathOperator{\Lip}{Lip}
\newcommand{\Wo}{{\raisebox{0.2ex}{\(\stackrel{\circ}{W}\)}}{}}

\newtheorem{theorem}{Theorem}
\newtheorem{lemma}{Lemma}
\newtheorem{proposition}{Proposition}
\newtheorem{corollary}{Corollary}

\theoremstyle{definition}
\newtheorem{definition}{Definition}
\newtheorem{remark}{Remark}
\newtheorem{example}{Example}
\newtheorem*{conjecture*}{Conjecture}

\newenvironment{enbibliography}{\vspace{-0.5cm}\begin{thebibliography}}{\end{thebibliography}}

\DeclareMathAlphabet{\mathbbold}{U}{bbold}{m}{n}

\begin{document} 
\title{On spectral asymptotics of the Sturm-Liouville problem with self-conformal singular weight}
\author{U.~R.~Freiberg\footnote{Institut f\"{u}r Stochastik und Anwendungen, Universit\"{a}t Stuttgart, Pfaffenwaldring 57, D-70569 Stuttgart, Germany. E-mail: freiberg@mathematik.uni-stuttgart.de.}, N.~V.~Rastegaev\footnote{Chebyshev Laboratory, St. Petersburg State University, 14th Line 29b, 199178 Saint-Petersburg, Russia. E-mail: rastmusician@gmail.com. Supported by the joint SPbU and DFG grant No.~6.65.37.2017}}
\maketitle

%\tableofcontents

%\pagebreak

\section{Introduction}

We consider spectral asymptotics for the Neumann problem
\begin{gather}
\left\{
\begin{split}\label{eq:1.1}
    &-y''=\lambda\mu y,\\ %\label{eq:1.2}
    &y'(0)=y'(1)=0,
\end{split}
\right.
\end{gather}
%or alternatively (but basically same)
%\begin{gather*}
%-\dfrac{d}{d\mu}\dfrac{d}{dx}y = \lambda y, \\
%y'(0)=y'(1),
%\end{gather*}
where the weight $\mu$ is a self-conformal measure on a line.
%(in particular, $\mu$ is singular with respect to the Lebesgue measure).

\begin{remark}
It is well known, that the change of the boundary conditions causes rank two perturbation of the quadratic form corresponding to the problem. It follows from the general variational theory 
(see \cite[\S10.3]{BS2}) that counting functions of the eigenvalues of boundary-value problems, related to the same equation, but different boundary conditions, cannot differ by more than 2.
\end{remark}
The problem of the eigenvalues asymptotic behavior for this problem goes back to the works of
M.~G.~Krein (see, for example, \cite{K}).

From \cite{BS1} it follows that if the measure $\mu$ contains absolutely
continuous component, its singular component does not influence
the main term of the spectral asymptotic.

In the case of singular measure $\mu$ it follows from early works by M.~G.~Krein, that  the counting function $N:(0,+\infty)\to\mathbb N$ of eigenvalues of the problem \eqref{eq:1.1} 
admits the estimate $o(\lda^{\frac{1}{2}})$ instead of the usual asymptotics $N(\lda)\sim C\lda^{\frac{1}{2}}$ in the case of measure containing a regular component. (see, e.g., \cite{KrKac} or \cite{McKeanRay}, and also \cite{B} for similar results for higher ever order operators and better lower bounds for eigenvalues for some special classes of measures).

The problem is comparatively well-studied in the case of a self-similar (self-affine) measure. Exact power exponent in the case of self-similar measure was obtained in \cite{F}. 
It is shown in \cite{SV} and \cite{KL} that the eigenvalues counting function of problem \eqref{eq:1.1} for the self-similar weight has the asymptotics
\begin{equation}\label{eq:mes_asymp}
	N(\lambda)=\lambda^D\cdot\bigl(s(\ln\lambda)+o(1)\bigr),
	\qquad \lambda\to+\infty,
\end{equation}
where $D\in(0,\frac{1}{2})$ and $s$ is a continuous $T$-periodic function, dependent on the choice of the weight $\mu$ (see also \cite{Naz} for similar asymptotics in the case of an arbitrary even order differential operator, and \cite{Uta} for similar results for problems containing two self-similar measures). A series of works \cite{VSh3, V, Rast} is dedicated to the fine properties of the function $s$ for incrementally generalized classes of self-similar measures.

The aim of this paper is to find the power exponent $D$ in the case of self-conformal measure with some special properties.

This paper has the following structure. \refpar{2} provides the necessary definitions of self-conformal measures, derives their properties and defines some restrictions. \refpar{3} introduces the formal boundary value problem and defines the spectrum under consideration. \refpar{4} gives the definition of the deformed self-similar measure establishes the spectral asymptotics for them and formulates the strong bounded distortion property, which is the main restriction on the self-conformal measures considered in this paper. \refpar{5} shows the connection between self-conformal measures with strong bounded distortion property and deformed self-similar measures, thus extending the spectral asymptotics to them.

We denote by $C$ different constants, the values of which are of no consequence.

\section{Self-conformal measures on a line}\label{par:2}
\noindent
Let $m\geq 2$.
\noindent We say $/ = (\varphi_1, \ldots, \varphi_m; \rho_1, \ldots, \rho_m)$ is a \textit{conformal iterated function system} on $[0, 1]$,
if:
\begin{enumerate}
\item $\varphi_i: [0,1] \to \varphi_i([0,1])$ is a $C^{1+\gamma}$ diffeomorphism for $\gamma>0$ and all $i=1,\ldots,m$.
\item $\varphi_i((0,1)) \subset (0,1)$ and $\varphi_i((0,1)) \cap \varphi_j((0,1)) = \varnothing$ for all $i, j = 1, \ldots, m$, $i \neq j$.
\item $0 < |\varphi'_i(x)| < 1$ for all $i=1, \ldots, m$ and all $x\in [0,1]$.
\item Positive numbers $\rho_i$ are such, that $\sum\limits_{i=1}^m \rho_i = 1$.
\end{enumerate}

\begin{remark}
Clearly, the conformity property is redundant on a line, but we use the same terminology as in multidimensional case for compatibility. For multidimensional definition on a smooth Riemannian manifold see \cite{Pat}. For more general definitions in a complete metric space, see \cite{H}.
\end{remark}

%and let
%$\{J_i = [a_i, b_i]\}_{i=1}^m$ be the sub-segments of $[0,1]$ without interior %intersections, $b_i \leq a_{i+1}$. Positive values
%$\{\rho_i\}_{i=1}^m$ are such that
%$\sum\limits_{i=1}^m\rho_i = 1$. 

%Denote by
%$\varphi_i: [0,1] \to J_i$ $C^{1+\gamma}$ diffeomorphisms between $[0,1]$ and
%$J_i$ and 
\noindent Without loss of generality we assume, that $\varphi_i$ are numbered in ascending order, i.e. $\varphi_i(x) \leq \varphi_{i+1}(y)$ for all $x, y \in [0,1]$, $i = 1,\ldots,m-1$. We define boolean values $e_i$ as follows:
\begin{equation*}
e_i =
\left\{
\begin{aligned}
&0, &&\quad \varphi_i(0) < \varphi_i(1), \\
&1, &&\quad \varphi_i(0) > \varphi_i(1).
\end{aligned}
\right.
\end{equation*}
As such, $e_i=1$ when $\varphi_i$ changes the orientation of the segment.

We define the operator $\mathcal{S}$ on the space $L_{\infty}[0,1]$ as follows:
\begin{equation*}
\mathcal{S}(f) = \sum\limits_{i=1}^m
\left(\chi_{\varphi_i([0,1])}(e_i + (-1)^{e_i}f\circ \varphi_i^{-1})+\chi_{\{x>\varphi_i(1-e_i)\}} \right)\rho_i.
\end{equation*}

\begin{lemma}\label{prop1}
%\textbf{\textsc{(see, e.g. \cite[Lemma 2.1]{Sh})}}
$\mathcal{S}$ is a contraction mapping on $L_{\infty}[0,1]$.
\end{lemma}
\begin{proof}
\begin{gather*}
\|\mathcal{S}(f_1) - \mathcal{S}(f_2)\|_\infty = \| \sum\limits_{i=1}^m
\left((f_1-f_2)\circ \varphi_i^{-1} \right)\chi_{\varphi_i([0,1])}\rho_i \|_\infty =
 \max_i\rho_i \cdot \| f_1-f_2 \|_\infty. 
\end{gather*}
We note, that $\max\limits_i\rho_i < 1$, which proves the lemma.
\end{proof}

Hence, by the Banach fixed-point theorem there exists a (unique) function $C\in L_{\infty}[0,1]$ such that $\mathcal{S}(C)=C$.
%\begin{definition}
%Such a function $C(t)$ will be called the \textit{generalized self-conformal Cantor %ladder} with $m$ steps.
%\end{definition}
Function $C(t)$ could be found as the uniform limit of the sequence $\mathcal{S}^k(f)$ for \mbox{$f(t)\equiv t$}, which allows us to assume that it is continuous and monotone, and also $C(0)=0$, $C(1)=1$.
The derivative of the function $C(t)$ in the sense of distributions is a measure $\mu$ without atoms, invariant with respect to $/$ in the sense of Hutchinson (see \cite{H}), i.e. it satisfies the relation
\begin{equation*}
\mu(E) = \sum\limits_{i=1}^m \rho_i \cdot \mu(\varphi_i^{-1}(E))
\end{equation*}
for any measurable set $E$.

\begin{definition}
We call $\mu$ \textit{self-conformal measure} and denote it
\[
\mu := \mu(\varphi_1, \ldots, \varphi_m; \rho_1, \ldots, \rho_m).
\]
\end{definition}

\begin{remark}
For a fixed measure $\mu$ the choice of iterated function system is not unique. Also, the definition does not require the function system to be conformal. For example, we will define measures using $W^1_\infty$ diffeomorphisms later. However, we call measure $\mu$ self-conformal only when it is possible to choose appropriate $C^{1+\gamma}$ diffeomorpisms to define it. 
\end{remark}

\begin{lemma}
Let us define
\[
\Phi(E) := \bigcup\limits_{i=1}^m \varphi_i(E),
\]
and let $|\Phi^k([0,1])|\to 0$ as $k\to\infty$. Then measure $\mu$ is singular with respect to Lebesgue measure.
\end{lemma}
\begin{proof}
It is obvious, that $\supp \mu \subset \Phi^k([0,1])$ for every $k$, thus $|\supp\mu| = 0$.
\end{proof}

%\begin{conjecture*}
%Measure $\mu$ is singular only if $|\Phi^k([0,1])|\to 0$ as $k\to\infty$.
%\end{conjecture*}
%\begin{proof}
%Not sure. Not even sure it's true. Something like this might be provable through %singularity criteria like \cite[\S 4]{VSh3}, but I'm not up for that yet. Upd: I don't %think it's true as is.
%\end{proof}

\begin{lemma}\label{lemma3}
Let $\Lip \sum\limits_{i=1}^m |\varphi_i-\varphi_i(0)| < 1$. Then $|\Phi^k([0,1])|\to 0$ as $k\to\infty$.
\end{lemma}
\begin{proof}
From $\Lip \sum\limits_{i=1}^m |\varphi_i-\varphi_i(0)| = \alpha < 1$ it follows by the definition of $\Phi$, that 
\[
|\Phi([0,a])| = \Lip \sum\limits_{i=1}^m |\varphi_i(a)-\varphi_i(0)| \leq \alpha |[0,a]|,
\]
thus for every measurable set $E$
$$|\Phi(E)| \leq \alpha |E|,$$ and thus $|\Phi^k([0,1])| \leq \alpha^k\to 0$ as $k\to\infty$.
\end{proof}

\begin{corollary}
Denote $\alpha_i := \Lip\varphi_i = \|\varphi'_i\|_\infty$ and let $\sum\limits_{i=1}^m \alpha_i < 1$. Then measure $\mu$ is singular with respect to Lebesgue measure.
\end{corollary}

\noindent Hereafter we always assume, that 
\begin{equation}\label{LipBound}
\sum\limits_{i=1}^m \alpha_i = \sum\limits_{i=1}^m \|\varphi'_i\|_\infty < 1.
\end{equation}
In particular, $[0,1]\setminus \Phi([0,1])$ contains at least one interval.

\begin{remark}
If all diffeomorphisms $\varphi_i$ are linear functions, we call $\mu$ \textit{self-similar measure}. For self-similar measures Lemma \ref{lemma3} means, that if $\Phi([0,1]) \neq [0,1]$, then $\mu$ is singular with respect to Lebesgue measure.
More general ways to construct self-similar functions on a line are described in
\cite{Sh}.
\end{remark}

\section{Sturm-Liouville problem with self-similar weight}
We consider the formal boundary value problem
\begin{gather}
\left\{
\begin{split}\label{eq:2.01}
    &-y''=\lambda\mu y,\\ 
    &y'(0)=y'(1)=0.
\end{split}
\right.
\end{gather}
We call the function $y \in W_2^1[0,1]$ its generalized solution if it satisfies the integral equation
$$ \int\limits_0^1 y'\eta'\, dx = 
\lda\int_0^1 y\eta \;d\mu(x)  $$
for any $\eta \in W_2^1[0,1]$. Substituting functions \mbox{$\eta \in \Wo_2^1[0,1]$} into the integral equation, we establish that the derivative $y'$ is a primitive of a singular measure without atoms $\lda\mu y$, thus $y \in C^1[0,1]$.

We denote by $\lambda_n(\mu)$ the eigenvalues of the problem \eqref{eq:2.01} numbered in ascending order, and by
\[
N(\lambda, \mu) := \#\{n: \lambda_n(\mu) < \lambda  \}
\]
their counting function.

\section{Deformed self-similar measures}

Consider $S_i: [0,1] \to I_i$ --- a set of affine (linear) contractions of $[0,1]$ onto non-intersecting subsegments $I_i$ of $[0,1]$. Denote by
\[
\mu_0 := \mu_0(S_1, \ldots, S_m; \rho_1, \ldots, \rho_m)
\] 
the self-similar measure generated by them. Let $g: [0,1] \to [0,1]$ be a $W^1_\infty$ diffeomorphism. 
\begin{definition}
We define a \textit{deformed self-similar measure} $\mu$ as
\[
\mu(E) := \mu_0(g(E))
\]
for every measurable set $E$. 
\end{definition}

\begin{lemma}
Let $\mu = \mu_0 \circ g$ be a deformed self-similar measure and let $g$ be a $C^{1+\gamma}$ diffeomorphism for some $\gamma>0$.
Then $\mu$ is a self-conformal measure. %if and only if.
\end{lemma}
\begin{proof}
It is clear, that
\[
\mu(E) = \mu_0(g(E)) = \sum\limits_{i=1}^m \rho_i \cdot \mu_0(S_i^{-1}(g(E)))=
\sum\limits_{i=1}^m \rho_i \cdot \mu\left((g^{-1}\circ S_i \circ g)^{-1}(E)\right),
\]
thus
\[
\mu = \mu(\varphi_1, \ldots, \varphi_m; \rho_1, \ldots, \rho_m),
\]
where $\varphi_i := g^{-1}\circ S_i \circ g$. It is clear, that $\varphi_i$ are $C^{1+\gamma}$ diffeomorphisms and the lemma is proved.

%We aim to show, that $\varphi_i$ are $C^{1+\gamma}$ diffeomorphisms if and only if $g$ is a $C^{1+\gamma}$ diffeomorphism.
\end{proof}

\begin{proposition}\label{propFuj}
\textbf{\textsc{\cite[Theorem 3.6]{F}}}
\[
N(\lambda, \mu_0) \asymp \lambda^D,
\]
i.e. there exist constants $C_1, C_2 > 0$, such that for all $\lambda \geq 0$
\[
C_1 \lambda^D \leq N(\lambda, \mu_0) \leq C_2 \lambda^D,
\]
where $D \in (0, \frac{1}{2})$ is the only solution of
\[
\sum\limits_{i=1}^m (\rho_i|I_i|)^D = 1.
\]
\end{proposition}

\begin{theorem}
Let $\mu$ be a deformed self-similar measure. Then
\[
N(\lambda, \mu) \asymp N(\lambda, \mu_0),%\lambda^D,
\]
i.e. there exist constants $C_1$, $C_2$, such that for all $\lambda \geq 0$
\[
C_1 N(\lambda, \mu_0) \leq N(\lambda, \mu) \leq C_2 N(\lambda, \mu_0).
\]
%where $D \in (0, \frac{1}{2})$ is a unique solution of
%\[
%\sum\limits_{i=1}^m (\rho_i|\varphi'_i(x_i)|)^D = 1.
%\]
%where $x_i$ is the fixed point of $\varphi_i$.
\end{theorem}
\begin{proof}
Consider $y \in W^1_2[0,1]$ and consider $z = y\circ g \in W^1_2[0,1]$. Note, that 
\begin{equation}\label{eq:4.1}
\int\limits_0^1 y^2(x) d\mu_0(x) = \int\limits_0^1 y^2(g(x)) d\mu(x)
= \int\limits_0^1 z^2(x) d\mu(x).
\end{equation}
Note, also, that by changing the variable $t:=g(x)$ we obtain
\begin{equation}\label{eq:4.1.1}
\int\limits_0^1 |z'(x)|^2 dx = 
\int\limits_0^1 |(y(g(x)))'|^2 dx = \int\limits_0^1 |y'(g(x))|^2 (g'(x))^2 dx = \int_0^1 |y'(t)| g'(g^{-1}(t)) dt.
\end{equation}
Since $g$ is  $W^1_\infty[0,1]$ diffeomorphism, there exist constants $q, Q > 0$, such that
\[
q < |g'(x)| < Q.
\] 
Hence, from \eqref{eq:4.1.1} it follows, that
\begin{equation}\label{eq:4.2}
q \int\limits_0^1 |y'(x)|^2 dx \leq \int\limits_0^1 |z'(x)|^2 dx \leq Q \int\limits_0^1 |y'(x)|^2 dx,
\end{equation}
thus, from \eqref{eq:4.1} and \eqref{eq:4.2}, using Courant–Fischer–Weyl min-max principle, we obtain 
\[
q \lambda_n(\mu_0) \leq \lambda_n(\mu) \leq Q \lambda_n(\mu_0),
\]
and 
\[
N(q\lambda, \mu_0) \leq N(\lambda, \mu) \leq N(Q\lambda, \mu_0).%\lambda^D,
\]
Note, that by the Proposition \ref{propFuj}
\[
N(q\lambda, \mu_0) \asymp N(\lambda, \mu_0), \quad 
N(Q\lambda, \mu_0) \asymp N(\lambda, \mu_0),
\]
and the theorem is proved.
\end{proof}

\begin{definition}
%Define the code space $\Sigma = \{1, \ldots, m\}^{\mathbb{N}}$ and l
Let's introduce the following notations: $$\Sigma_k = \{1, \ldots, m\}^{k}, \quad \Sigma_* = \bigcup\limits_{i=0}^\infty \Sigma_i,$$ 
for a word $w=(i_1, i_2, \ldots, i_k)\in \Sigma_k$ we say $|w| = k$, denote 
\[
\varphi_w = \varphi_{i_1}\circ \varphi_{i_2} \circ \ldots \circ \varphi_{i_k}
\] 
and denote by $x_w$ the unique fixed point of $\varphi_w$: $$\varphi_w(x_w) = x_w.$$
We will also use notation $\varphi^{[k]}$ for the composition of $k$ instances of function $\varphi$.

\end{definition}

\begin{remark}
The conformal iterated function system fulfils the \textit{bounded distortion property} (see \cite[Lemma 2.1]{Pat}), i.e. there exists a constant $C\geq 1$, such that
\[
C^{-1} \leq \left| \dfrac{\varphi'_w(x)}{\varphi'_w(y)} \right| \leq C
\]
for all $x, y \in [0,1]$ and all $w \in \Sigma_*$.

For a deformed self-similar measure we have
\[
\varphi_w = g^{-1}\circ S_w \circ g,
\] 
for every word $w \in \Sigma_*$, thus
\begin{equation}\label{eq:4.3}
\varphi'_w(x) = \dfrac{g'(x)}{g'(\varphi_w(x))} |S_w([0,1])|
\end{equation}
for almost every $x \in [0,1]$. Note, that for $w = (i_1, i_2, \ldots, i_k) \in \Sigma_k$
\[
|S_w([0,1])| = |I_{i_1}|\cdot |I_{i_2}| \cdot \ldots \cdot |I_{i_k}|,
\]
does not depend on the order of elements of $w$, thus for $C=Q^2/q^2$ we have
\[
C^{-1} \leq \left| \dfrac{\varphi'_w(x)}{\varphi'_{\sigma w}(y)} \right| \leq C
\]
for all $x, y \in [0,1]$, all $w \in \Sigma_*$ and all permutations $\sigma \in \mathcal{S}_{|w|}$.
We call this \textit{strong bounded distortion property}. We are going to show in Sect.~5, that this property is sufficient to prove, that self-conformal measure is a deformed self-similar measure.
\end{remark}

%\begin{definition} We denote by $x_w$ the unique fixed point of $\varphi_w$.
%\[
%\varphi_i(x_i) = x_i.
%\]
%\end{definition}
\begin{remark}
It follows from \eqref{eq:4.3}, that if $g'$ exists at the point $x_i$, then
\[
|\varphi_i'(x_i)| = |I_i|.
\]
If $g'(x_i)$ is not defined, then
\[
|I_i| = \lim\limits_{k\to\infty} \sqrt[k]{\left|\varphi_i^{[k]}([0,1])\right|},
\]
so if $\varphi_i$ are $C^{1+\gamma}$ diffeomorphisms, then it is easy to see, that for all $x\in[0,1]$
\begin{equation}\label{eq:4.4}
|\varphi'_i(\varphi_i^{[k]}(x)) - \varphi'_i(x_i)| \leq C\alpha_i^{k\gamma},
\end{equation}
thus
\[
\lim\limits_{k\to\infty} \sqrt[k]{\left|\dfrac{\varphi_i^{[k]}([0,1])}{(\varphi'_i(x_i))^k}\right|} = 1,
\]
and, yet again, we have 
\[
|\varphi_i'(x_i)| = |I_i|.
\]
%If $\varphi'_i$ exists at the point $x_i$ for all $i$, then $D$ is the only solution %of
%\[
%\sum\limits_{i=1}^m (\rho_i|\varphi'_i(x_i)|)^D = 1.
%\]
\end{remark}

\begin{corollary}
Let $\mu = \mu_0 \circ g$ be a deformed self-similar measure and let $g$ be a $C^{1}$ diffeomorphism. Then
\[
N(\lambda, \mu) \asymp \lambda^D,
\]
where $D \in (0, \frac{1}{2})$ is the only solution of
\[
\sum\limits_{i=1}^m (\rho_i|\varphi'_i(x_i)|)^D = 1.
\]
\end{corollary}

\section{Deformation construction}

Consider a self-conformal measure
\[
\mu = \mu(\varphi_1, \ldots, \varphi_m; \rho_1, \ldots, \rho_m),
\]
that satisfies the relation \eqref{LipBound} and the strong bounded distortion property, and consider a self-similar measure
\[
\mu_0 := \mu_0(S_1, \ldots, S_m; \rho_1, \ldots, \rho_m),
\]
that has the same structure as $\mu$ ($S_i$ changes orientation if and only if $\varphi_i$ changes orientation; $S_i([0,1])$ and $S_{i+1}([0,1])$ touch if and only if $\varphi_i([0,1])$ and $\varphi_{i+1}([0,1])$ touch).

This section is concerned with two questions:
\begin{itemize}
\item If there exists a mapping $g$, such that $\mu = \mu_0 \circ g$.
\item If there exist $S_1$, $S_2$, \ldots, $S_m$, such that $g$ is a $W^1_\infty$-diffeomorphism.
\end{itemize}

\subsection{Construction}

%Let's define for a word $w=(i_1, i_2, \ldots, i_k)$
%\[
%\varphi_w = \varphi_{i_1}\circ \varphi_{i_2} \circ \ldots \circ \varphi_{i_k},
%\] 
%and l
For $i = 1,\ldots,m-1$ denote by $U_i$ the intermediate interval (possibly empty) between $\varphi_i([0,1])$ and $\varphi_{i+1}([0,1])$, i.e.
\[
U_i = (c_i, d_i) := \{ x\in[0,1] : \forall y\in[0,1] \;\; \varphi_i(y) < x < \varphi_{i+1}(y)  \}.
\]
It is clear, that 
\[
[0, 1] \setminus \supp \mu = \bigcup\limits_{i=1}^{m-1} \bigcup\limits_{w\in\Sigma_*} \varphi_w(U_i).
\]
Let's construct mapping $g$ explicitly by defining
\begin{equation}\label{eq:5.1.a}
\forall w \in \Sigma_* \qquad g(\varphi_w(0)) = S_w(0), \quad g(\varphi_w(1)) = S_w(1),
\end{equation}
connecting the dots linearly on each interval $\varphi_w(U_i)$, and thus everywhere outside $\supp \mu$,
and extending the definition continuously onto $\supp \mu$.

By this definition,
\[
\mu = \mu_0 \circ g,
\]
and $g'$ exists almost everywhere by Lebesgue measure.
\begin{lemma}
Denote
\[
\widetilde\varphi_i := g^{-1}\circ S_i \circ g.
\]
Then $\widetilde\varphi_i$ is a $W^1_\infty$ diffeomorphism, $\widetilde\varphi_i = \varphi_i$ on $\supp \mu$,
\[
\forall w\in\Sigma_* \quad \widetilde\varphi_w(0) = \varphi_w(0), \quad \widetilde\varphi_w(1) = \varphi_w(1), 
\quad \widetilde\varphi_w(x_w) = x_w,
\]
and
\[
\mu = \mu(\widetilde\varphi_1, \ldots, \widetilde\varphi_m; \rho_1, \ldots, \rho_m).
\]
\end{lemma}
\begin{proof}
Function $g$ is continuous and strictly monotonous by definition, so $g^{-1}$ is also continuous, and so is $\widetilde\varphi_i$ for every $i$.
Using \eqref{eq:5.1.a} we obtain for all $w\in\Sigma_*$
\[
\widetilde\varphi_w(0) = g^{-1}(S_w(g(0))) = \varphi_w(0), \quad 
\widetilde\varphi_w(1) = g^{-1}(S_w(g(1))) = \varphi_w(1),
\]
thus $\widetilde\varphi_i = \varphi_i$ on $\supp \mu$, since both functions are continuous and every point of $\supp \mu$ is a limit point of the set $\{ \varphi_w(0), w\in\Sigma_* \}$. Fixed point $x_w$ is also a limit point of the set $\{ \varphi_w(0), w\in\Sigma_* \}$, since $x_w = \lim\limits_{k\to\infty} \varphi_w^{[k]}(0)$.
Outside $\supp \mu$ functions $\widetilde\varphi_i$ are linear on every interval, moreover, for all $i=1,\ldots,m$, $j=1,\ldots,m-1$, $v,w\in\Sigma_*$ we have
\begin{equation}\label{DerivInt}
|\widetilde\varphi'_v| \equiv \dfrac{|\varphi_v(\varphi_w(U_j))|}
{|\varphi_w(U_j)|} =
\dfrac{1}{|\varphi_w(U_j)|} \int\limits_{\varphi_w(U_j)} |\varphi'_v(t)| dt
\quad \text{ on } \varphi_w(U_j),
\end{equation}
%which, given $\varphi_i\in C^{1+\gamma}$, gives us an estimate
%\begin{equation}\label{DerivEst}
%|\widetilde\varphi'_i - \varphi'_i| \leq C |\varphi_w(U_j)| \quad \text{ on } \varphi_w(U_j)
%\end{equation}
%for some $C>0$, and also
which gives us the relation
\[
\|\widetilde\varphi'_i\|_\infty \leq \|\varphi'_i\|_\infty, 
\]
thus $\widetilde\varphi_i \in W^1_\infty$. Similarly, $\widetilde\varphi^{-1}_i \in W^1_\infty$, thus $\widetilde\varphi_i$ is a $W^1_\infty$ diffeomorphism.

\end{proof}

\subsection{Smoothness}

We want to choose $S_i$ in such a way, that $g$ turns out to be $W^1_\infty$-diffeomorphism. Let's define
\begin{equation}\label{eq:5.2.a}
S_i(x) = c_i + \varphi'_i(x_i) x,
\end{equation}
where $c_i$ are arbitrary, but chosen in such a way, that $\mu$ and $\mu_0$ have the same structure as described at the beginning of Sect.~5.
%We denote by $\varphi_i^{[k]}$ the composition of $k$ instances of $\varphi_i$. 

\begin{lemma}\label{lemma6}
Let $(\varphi_1,\ldots,\varphi_m)$ satisfy the strong bounded distortion property. Then for every $w = (i_1,\ldots,i_k)\in\Sigma_*$ and every $x\in[0,1]$
\begin{equation}\label{StrongerBound}
C^{-1} \leq \left|  \dfrac{\widetilde\varphi'_w(x)}{\varphi'_{i_1}(x_{i_1})\cdot\ldots\cdot\varphi'_{i_k}(x_{i_k})}  \right| \leq C
\end{equation}
for some $C>0$.
\end{lemma}
\begin{proof}
From \eqref{DerivInt} we obtain, that since $\widetilde\varphi'_w$ is an average of $\varphi'_w$ on intervals outside $\supp\mu$, then for every $x\in[0,1]$ there exist some $x_1, x_2 \in [0,1]$, such that
\[
\varphi'_w(x_1) \leq \widetilde\varphi'_w(x) \leq \varphi'_w(x_2).
\]
Thus, it is sufficient to prove, that
\[
C^{-1} \leq \left|  \dfrac{\varphi'_w(x)}{\varphi'_{i_1}(x_{i_1})\cdot\ldots\cdot\varphi'_{i_k}(x_{i_k})}  \right| \leq C
\]
for some $C>0$, every $w = (i_1,\ldots,i_k)\in\Sigma_*$ and every $x\in[0,1]$. Using strong bounded distortion property we conclude, that without loss of generality, we could assume, that $i_1 \leq i_2 \leq \ldots \leq i_k$. Then
\[
\varphi_w = \varphi_1^{[k_1]} \circ \ldots \circ \varphi_m^{[k_m]}
\]
for certain values of $k_1, \ldots, k_m$. By the estimate \eqref{eq:4.4}, for every $x\in[0,1]$ we have
\[
\prod\limits_{j=0}^{k_i-1}(1 - C\alpha_i^{j\gamma}) \leq \left| \dfrac{(\varphi_i^{[k_i]})'(x)}{(\varphi_i(x_i))^{k_i}} \right|  \leq \prod\limits_{j=0}^{k_i-1}(1 + C\alpha_i^{j\gamma}).
\]
The products converge, since $\alpha_i < 1$, $\gamma>0$. Now it is sufficient to note, that
\[
\left|  \dfrac{\varphi'_w(x)}{\varphi'_{i_1}(x_{i_1})\cdot\ldots\cdot\varphi'_{i_k}(x_{i_k})}  \right| = \left| \dfrac{(\varphi_1^{[k_1]})'(\xi_1)}{(\varphi_1(x_1))^{k_1}} \right| \cdot\ldots\cdot\left| \dfrac{(\varphi_m^{[k_m]})'(\xi_m)}{(\varphi_m(x_m))^{k_m}} \right|,
\]
where $\xi_i := \varphi_{i+1}^{[k_{i+1}]}\circ\ldots\circ\varphi_m^{[k_m]}(x)$, is a product of $m$ bounded terms. 
\end{proof}

\begin{lemma}\label{lemma7}
Let self-conformal measure 
\[
\mu = \mu(\varphi_1, \ldots, \varphi_m; \rho_1, \ldots, \rho_m)
\]
satisfy the relation \eqref{LipBound} and the strong bounded distortion property. Then $\mu$ is a deformed self-similar measure, i.e. $\mu = \mu_0 \circ g$, where
\[
\mu_0 := \mu_0(S_1, \ldots, S_m; \rho_1, \ldots, \rho_m),
\]
$S_i$ are defined in \eqref{eq:5.2.a} and $g$ is a $W^1_\infty$ diffeomorphism.
\end{lemma}
\begin{proof}
By definition
\begin{equation}\label{eq:5.2.b}
\widetilde\varphi'_w(x) = \dfrac{g'(x)}{g'(\widetilde\varphi_w(x))} |S_w([0,1])|
\end{equation}
far a.e. $x\in[0,1]$ and all $w\in\Sigma_*$. From \eqref{eq:5.2.a}
\[
|S_w([0,1])| = |\varphi'_{i_1}(x_{i_1})|\cdot\ldots\cdot|\varphi'_{i_k}(x_{i_k})|,
\]
thus
\[
\dfrac{g'(x)}{g'(\widetilde\varphi_w(x))} = \dfrac{\widetilde\varphi'_w(x)}{\varphi'_{i_1}(x_{i_1})\cdot\ldots\cdot\varphi'_{i_k}(x_{i_k})},
\]
and by Lemma \ref{lemma6} we have
\[
C^{-1} \leq \left| \dfrac{g'(x)}{g'(\widetilde\varphi_w(x))} \right| \leq C.
\]
Note, that for all $x,y\in U:=\bigcup\limits_{i=1}^{m-1} U_i$
\[
C_1^{-1} \leq \left| \dfrac{g'(x)}{g'(y)} \right| \leq C_1
\]
for some $C_1>0$, since $g'$ has only finite number of values inside $U$. Thus, since
\[
\left| \dfrac{g'(\widetilde\varphi_{w_1}(x))}{g'(\widetilde\varphi_{w_2}(y))} \right| = \left| \dfrac{g'(\widetilde\varphi_{w_1}(x))}{g'(x)} \right| \cdot 
\left| \dfrac{g'(y)}{g'(\widetilde\varphi_{w_2}(y))} \right| \cdot 
\left| \dfrac{g'(x)}{g'(y)} \right|,
\]
we conclude, that
\[
(C^2C_1)^{-1} \leq \left| \dfrac{g'(x)}{g'(y)} \right| \leq C^2C_1
\]
for all $x, y \in \bigcup\limits_{i=1}^{m-1} \bigcup\limits_{w\in\Sigma_*} \varphi_w(U_i) = [0, 1] \setminus \supp \mu$, and $g$ is a $W^1_\infty$ diffeomorphism.
\end{proof}

\begin{theorem}
Let self-conformal measure 
\[
\mu = \mu(\varphi_1, \ldots, \varphi_m; \rho_1, \ldots, \rho_m)
\]
satisfy the relation \eqref{LipBound} and the strong bounded distortion property. Then
\[
N(\lambda, \mu) \asymp \lambda^D,
\]
where $D \in (0, \frac{1}{2})$ is the only solution of
\[
\sum\limits_{i=1}^m (\rho_i|\varphi'_i(x_i)|)^D = 1.
\]
\end{theorem}

Every conformal iterated function system fulfils the bounded distortion property (see \cite[Lemma 2.1]{Pat}), but 
not every conformal iterated function system fulfils the strong bounded distortion property.

\begin{example}
Denote $w_1 = (1, 2, 1, 2)$ and $w_2 = (2, 1, 1, 2)$.
Consider $C^{1+\gamma}$ diffeomorphisms $\varphi_1$ and $\varphi_2$, such that for some $\varepsilon > 0$ we have
\begin{equation}\label{eq:ex}
\varphi'_{w_1}(x) < (1-\varepsilon)\varphi'_{w_2}(y)
\end{equation}
for every $x\in\varphi_{w_1}([0,1])$, $y\in\varphi_{w_2}([0,1])$. We could, for example, choose $\varphi_1(x) = ax$ for some $0<a<1$, and construct $\varphi_2(x)$ in such a way, that
\begin{align*}
&\sqrt{1-\varepsilon}\cdot \varphi'_2(y) > \varphi'_2(x), &\quad& \forall x\in\varphi_1([0,1]), \; y\in\varphi_2([0,1]), \\
&\sqrt{1-\varepsilon}\cdot \varphi'_2(y) > \varphi'_2(x), &\quad& \forall x\in\varphi_1\circ\varphi_2([0,1]), \; y\in\varphi_1\circ\varphi_1\circ\varphi_2([0,1]).
\end{align*}
From this condition \eqref{eq:ex} follows, since
\[
\varphi'_{w_1}(x) = a^2 \varphi'_2(\varphi_1(\varphi_2(x)))\varphi'_2(x) < 
a^2 (1-\varepsilon) \varphi'_2( \varphi_1(\varphi_1(\varphi_2(y))) ) \varphi'_2(y) = (1-\varepsilon)\varphi'_{w_2}(y)
\]
for every $x\in\varphi_{w_1}([0,1])$, $y\in\varphi_{w_2}([0,1])$. Note, that under this conditions
\[
\left| \dfrac{\left(\varphi_{w_1}^{[k]}\right)'(x)}{\left(\varphi_{w_2}^{[k]}\right)'(y)} \right| < (1-\varepsilon)^k
\]
for every $k\in\mathbb{N}$, $x\in\varphi_{w_1}([0,1])$, $y\in\varphi_{w_2}([0,1])$, which contradicts the strong bounded distortion property.

\end{example}

\bigskip
\bigskip

\begin{enbibliography}{99}
\addcontentsline{toc}{section}{References}

\bibitem{BS2} Birman~M.~Sh., Solomyak~M.~Z. \emph{Spectral theory of self-adjoint operators in Hilbert space} // 
Ed.2, Lan' publishers. --- 2010. (in Russian); English translation of the 1st ed: Mathematics and its Applications (Soviet Series). D. Reidel Publishing Co., Dordrecht, 1987.

\bibitem{K} Krein~M.~G. \emph{Determination of the density of the symmetric inhomogeneous string by spectrum}//
Dokl. Akad. Nauk SSSR --- 1951. --- V.~76, N.~3. --- P.~345-348. (in Russian)

\bibitem{BS1} Birman~M.~Sh., Solomyak~M.~Z. \emph{Asymptotic behavior of the spectrum of weakly polar integral operators}// Mathematics of the USSR-Izvestiya. --- 1970. --- V.~4, N.~5. --- P.~1151-1168.

\bibitem{KrKac} Krein~M.~G., Kac~I.~S. \emph{A discreteness criterion for the spectrum of a singular string}// Izvestiya Vuzov Matematika. --- 1958. --- N.~2. --- P.~136-153. (in Russian)

\bibitem{McKeanRay} McKean~H.~P., Ray~D.~B. \emph{Spectral distribution of a differential operator}// Duke Mathematical Journal. --- 1962. --- V.~29. --- N.~2. --- P.~281-292.

\bibitem{B} Borzov~V.~V. \emph{On the quantitative characteristics of singular measures}// Problems of math. physics. --- 1970. --- V.~4.
--- P.~42-47. (in Russian)

\bibitem{F} Fujita~T. \emph{A fractional dimention, self-similarity
and a generalized diffusion operator} // Taniguchi Symp. PMMP. Katata.
--- 1985. --- P.~83-90.

\bibitem{SV} Solomyak~M., Verbitsky~E. \emph{On a spectral problem related to
self-similar measures}
// Bull. London Math.~Soc. --- 1995. --- V.~27, N~3. ---
P.~242-248.

\bibitem{KL} Kigami~J., Lapidus~M.~L. \emph{Weyl’s problem for the spectral 
distributions of Laplacians on p.c.f. self-similar fractals}// Comm. Math. Phys. ---
1991. --- V.~158. --- P.~93–125.

\bibitem{Naz} Nazarov~A.~I. \emph{Logarithmic \(L_2\)-small ball asymptotics with respect to self-similar measure for some Gaussian processes}// Journal of Mathematical Sciences (New York). --- 2006. --- V.~133, N.~3. --- P.~1314-1327.

\bibitem{Uta} Freiberg~U.~R. \emph{A Survey on Measure Geometric Laplacians on Cantor Like Sets} // Arabian Journal for Science and Engineering. --- 2003. --- V.~28. --- N.~1C. --- P.~189-198.

\bibitem{VSh3} Vladimirov~A.~A., Sheipak~I.~A. \emph{On the Neumann Problem for the Sturm–Liouville Equation with Cantor-Type Self-Similar Weight}// Functional Analysis and Its Applications. --- 2013. --- V.~47, N.~4. --- P.~261-270.

\bibitem{V} Vladimirov~A.~A. \emph{Method of oscillation and spectral problem for four-order differential operator with self-similar weight}//
St. Petersburg Math. J. --- 2016. --- V.~27. --- N.~2. --- P.~237-244.

\bibitem{Rast} Rastegaev~N.~V. \emph{On spectral asymptotics of the Neumann problem for the Sturm–Liouville equation with self-similar generalized Cantor type weight}// J. Math. Sci. (N. Y.). --- 2015. --- V.~210. --- N.~6. --- P.~814-821.

\bibitem{Sh} Sheipak~I.~A. \emph{On the construction and some properties of self-similar functions in the spaces \(L_p[0,1]\)}//  Mathematical Notes. --- 2007. ---
 V.~81, N.~6. --- P.~827–839.

\bibitem{H} Hutchinson~J.~E. \emph{Fractals and self similarity}//
Indiana Univ. Math. J. --- 1981. --- V.~30, N~5. --- P.~713-747.

\bibitem{Pat} Patzschke~N., \emph{Self-Conformal Multifractal Measures} // Advances in Applied Mathematics. --- 1997. --- V.~19. --- P.~486-513.

\end{enbibliography}

\end{document}